\documentclass[12pt,a4paper]{amsart}
\usepackage{amsfonts}
\usepackage{amsthm}
\usepackage{amsmath}
\usepackage{amscd}
\usepackage[latin2]{inputenc}
\usepackage{t1enc}
\usepackage[mathscr]{eucal}
\usepackage{indentfirst}
\usepackage{graphicx}
\usepackage{graphics}
\usepackage{pict2e}
\usepackage{epic}
\numberwithin{equation}{section}
\usepackage[margin=2.9cm]{geometry}
\usepackage{epstopdf} 

\usepackage{amsthm}
\usepackage{amssymb}
\usepackage{amsmath}
\usepackage{ytableau}
\usepackage{tikz-cd}
\usepackage{mathtools}
\usepackage{enumerate}
\usepackage{wasysym}
\usepackage[framemethod=tikz]{mdframed}
\usepackage{titlesec}
\usepackage{amsaddr}

\theoremstyle{definition}
\newtheorem{mydef}{Definition}
\newtheorem{mythe}[mydef]{Theorem}

\newtheorem{cor}[mydef]{Corollary}
\newtheorem{pro}[mydef]{Proposition}
\newtheorem{lem}[mydef]{Lemma}

\newcommand\chaptername{Chapter}
\titleformat{\section}[display]
  {\normalfont\huge\bfseries\sffamily}
  {\chaptername\ \thesection}
  {20pt}
  {\Huge}
\titlespacing*{\section}
  {0pt}{50pt}{40pt}

\newcommand\Q{\mathbb{Q}}
\newcommand\Z{\mathbb{Z}}
\newcommand\N{\mathbb{N}}

\newcommand\F{\mathbb{F}}

\newcommand{\M}{\mathcal{M}}

\renewcommand{\l}{\lambda}

\newcommand{\Modd}[1]{{#1}\textrm{-Mod}}

\DeclareMathOperator{\im}{im}

\DeclareMathOperator{\Hom}{Hom}

\DeclareMathOperator{\Ext}{Ext}
\DeclareMathOperator{\Tor}{Tor}
\DeclareMathOperator{\Gh}{Gh}
\DeclareMathOperator{\ccs}{ccs}

\title{Cohomology of Burnside Rings}
\author{Benen Harrington}
\address{Department of Mathematics, University of York, York YO10 5DD, UK}
\email{bh885@york.ac.uk}

\begin{document}
\maketitle
\begin{center}{\emph{Abstract.}}\end{center}

\noindent Let $G$ be a finite group and $A(G)$ its Burnside ring. For $H \subset G$ let $\Z_H$ denote the $A(G)$-module corresponding to the mark homomorphism associated to $H$. When the order of $G$ is square-free we give a complete description of the $A(G)$-modules $\Ext^l_{A(G)}(\Z_H, \Z_J)$ and $\Tor^{A(G)}_l(\Z_H, \Z_J)$ for any $H, J \subset G$ and $l \geq 0$. We show that if the order of $G$ is not square-free then there exist $H, J \subset G$ such that $\Ext^l_{A(G)}(\Z_H, \Z_J)$ and $\Tor^{A(G)}_l(\Z_H, \Z_J)$ have unbounded rank as finite groups.
\newline

\begin{center}{\emph{Preliminaries.}}\end{center}

\noindent Let $G$ be a finite group. The isomorphism classes of finite $G$-sets form a commutative semi-ring where addition is given by disjoint union and multiplication is given by cartesian product. The Burnside ring $A(G)$ is the Grothendieck ring associated to this semi-ring. For a finite $G$-set $X$, we write $[X]$ for the corresponding isomorphism class in $A(G)$. We first recall some facts about the Burnside ring, see \cite{tomdieck} (Chapter 1) for proofs and further details.

Given a subgroup $H \subset G$, the set of left cosets $G/H$ has the natural structure of a $G$-set, where for $g, g' \in G$ and $g'H \in G/H$, we have $g\cdot g'H = gg'H$. Given $J \subset G$, the $G$-sets $G/H$ and $G/J$ are isomorphic if and only if $H$ is conjugate to $J$ in $G$. Each transitive $G$-set $X$ is isomorphic to $G/H$ for some $H \subset G$, and the Burnside ring is free on the set of isomorphism classes of transitive $G$-sets. Write $\ccs(G)$ for the set of conjugacy classes of subgroups of $G$. For $H \subset G$, write $(H)$ for the conjugacy class of subgroups to which $H$ belongs.

For a subgroup $H \subset G$ and finite $G$-set $X$, let $X^H$ denote the subset of $X$ of points fixed by $H$. The mark homomorphism $\pi_H:A(G) \to \Z$ associated to $H$ is defined by putting $\pi_H([X]) = |X^H|$ for each finite $G$-set $X$ and extending to $A(G)$. Write $\Z_H$ for the left $A(G)$-module structure on $\Z$ defined by $a\cdot n = \pi_H(a)n$ for $a \in A(G)$ and $n \in \Z$. 

\begin{lem}\label{first}
\begin{enumerate} Let $H, J$ be subgroups of $G$.
\item $\pi_H([G/H]) = [N_GH:H]$.
\item $\pi_J([G/H]) \neq 0$ if and only if $J$ is conjugate to a subgroup of $H$.
\item $\pi_H = \pi_J$ if and only if $(H) = (J)$.
\end{enumerate}
\end{lem}

For each $(H) \in \ccs(G)$ we then have a well-defined homomorphism $\pi_{(H)} = \pi_H$. We have a  homomorphism of rings 
\[
\pi:A(G) \to \prod_{(H) \in \ccs(G)}\Z
\]
where for $a \in A(G)$, 
\[
a \mapsto (\pi_{(H)}(a))_{(H) \in \ccs(G)}.
\]
This is an embedding of rings. The ring $\prod_{(H)\in\ccs(G)}\Z$ is called the ghost ring of $G$. Most of our results on the Burnside ring hold more generally for arbitrary subrings of the ghost ring, and in the next section we introduce the notion of a $B$-ring in order to provide the appropriate setting for stating these results.

Let $R$ be a commutative ring. Let $\textrm{Ab}$ be the category of abelian groups and $\Modd{R}$ the category of left $R$-modules. For $R$-modules $M$, $N$, the functors $\Hom_{R}(M, -):\Modd{R} \to \textrm{Ab}$ and $\Hom_{R}(-, N):\Modd{R} \to \textrm{Ab}$ are left exact. For $l$ a non-negative integer, we write $\Ext^l_{R}(M, N)$ for the $l$th right derived functor of $\Hom_{R}(M, -)$ applied to the module $N$, or equivalently for the $l$th right derived functor of $\Hom_{R}(-,N)$ applied to the module $M$. The functor $-\otimes_{R} N:\Modd{R} \to \textrm{Ab}$ is right-exact, and we write $\Tor_l^{R}(M, N)$ for the $l$th left derived functor of $-\otimes_{R}N$ applied to the module $M$. Since $R$ is a commutative ring, each $\Ext^l_{R}(M, N)$ and $\Tor^{R}_l(M, N)$ is naturally endowed with the structure of a left $R$-module. In what follows all modules over a commutative ring are left modules.
\newline

\begin{center}\emph{$B$-rings.}\end{center}

\noindent For $I$ a finite set, define $\Gh(I) = \prod_{i \in I} \Z$. Let $R$ be a subring of $\Gh(I)$ and for each $i \in I$ let $\pi_i$ be the corresponding projection $R \to \Z$. For $r \in R$, write $r(i)$ for $\pi_i(r)$. 

\begin{mydef}\label{sepcond}
Say that $R \subset \Gh(I)$ is a $B$-ring if for each distinct pair $i, j \in I$ there exists an $r \in R$ with $r(i) \neq 0$ and $r(j) = 0$.
\end{mydef}

If some subring $R \subset \Gh(I)$ is not a $B$-ring, with $i, j \in I$ a pair for which the above condition fails, then it is clear that $r(i) = r(j)$ for all $r \in R$. Then $R$ is isomorphic to the ring $S \subset \Gh(I -\{j\})$ obtained by omitting the factor corresponding to $j$. Repeating this process if necessary we obtain a subset $I' \subset I$ and a $B$-ring $R' \subset \Gh(I')$ with $R$ isomorphic to $R'$.

We give an intrinsic definition of these rings as follows.

\begin{pro}\label{intrinsic}
A ring $S$ is isomorphic to a $B$-ring $R \subset \Gh(I)$ for some finite set $I$ if and only if $S$ is a commutative ring which is of finite rank and torsion-free as a $\Z$-module, with $\Q \otimes_\Z R$ a product of $|I|$ copies of $\Q$.
\end{pro}

\begin{proof}
If $R \subset \Gh(I)$ is a $B$-ring then it is certainly commutative and torsion-free, since $\Gh(I)$ is. As a $\Z$-module $\Gh(I)$ is finitely generated, so $R$ is of finite rank. For each pair $i, j$ of distinct elements of $I$, let $r_{i, j}$ be an element of $R$ satisfying $r(i) \neq 0$ and $r(j) = 0$. Then putting $s_i = \prod_{j \neq i}r_{i, j}$ for each $i \in I$, we have $s_i(j) \neq 0$ if and only if $i = j$. Let $N = \prod_{i\in I}s_i(i)$ and $N_i = N / s_i(i)$. For $i \in I$ write $e_i$ for the corresponding primitive idempotent of $\Gh(I)$. Then
\[
N\cdot e_i = N_is_i \in R,
\]
and so $N\cdot \Gh(I) \subset R \subset \Gh(I)$. Hence 
\[
\Q\otimes_\Z R \simeq \Q \otimes_\Z \Gh(I) \simeq \prod_{i \in I} \Q,
\]
i.e.\ $\Q \otimes R$ is isomorphic to a product of $|I|$ copies of $\Q$.

Suppose $S$ is a commutative ring which is of finite rank and torsion-free as a $\Z$-module, with $\Q\otimes_\Z S$ a product of finitely copies of $\Q$. Then we have an isomorphism
\[
\theta: \Q\otimes S \to \prod_{i \in I'} \Q
\]
for some finite indexing set $I'$. 

Since $S$ is torsion-free, $\Q \otimes S$ contains a copy of $S$ as the subring $1 \otimes S \subset \Q \otimes S$. Denote the image $\theta(1\otimes S)$ by $S' \subset \prod_{i \in I'}\Q$, and for each $i \in I'$ let $\hat \pi_i$ denote the projection map $\prod_{i \in I'}\Q \to \Q$ onto the $i$th factor. Write $\pi_i$ for the restriction of $\hat\pi_i$ to $S'$. Since $S'$ is of finite rank as a $\Z$-module we must have that $\pi_i(S) \subset \Z \subset \Q$ for each $i\in I'$. We can then regard $S'$ as sitting inside $\Gh(I') = \prod_{i \in I'} \Z \subset \prod_{i \in I'} \Q$. We claim that this embedding defines a $B$-ring. For $s \in S$, write $s'$ for the element $\theta(1\otimes s)$ of $S'$. It remains to show that for each distinct pair $i, j \in I'$ we can find an element $s \in S$ such that $\pi_i(s') \neq 0$ and $\pi_j(s') = 0$. 

Let $f_1, \ldots , f_n$ be the primitive idempotents of $\Gh(I')$, and note that $\hat\pi_j(f_i) = 1$ if $j = i$ and $\hat\pi_j(f_i) = 0$ otherwise. For each $i \in I$, we have $f_i = \theta(q_i \otimes t_i)$ for some $q_i \in \Q$ and $t_i \in S$. Then $t'_i = \theta(1\otimes t_i) = (1/q_i)\cdot \theta(q_i\otimes t_i) = (1/q_i) f_i$, and so $\pi_i(t'_i) = 1/q_i \neq 0$ and $\pi_j(t'_i) = 0$ for each $j \neq i$. Thus $t'_i$ satisfies the condition of Definition \ref{sepcond} for any $j \neq i$, and the embedding $S' \subset \Gh(I')$ defines a $B$-ring.
\end{proof}

Let $R \subset \Gh(I)$ be a $B$-ring. For each $i \in I$ we have an $R$-module $\Z_i$ defined by letting $R$ act on the set $\Z$ by $r\cdot n = r(i) n$ for $r \in R$ and $n \in \Z$.

\begin{lem}
The $R$-modules $
\Ext^l_{R}(\Z_i, \Z_j)
$ and $\Tor^R_l(\Z_i, \Z_j)$ 
are finite for any $l \geq 1$ and $i, j \in I$.
\end{lem}

\begin{proof}
Note that $R_\Q := \Q\otimes_\Z R$ is semisimple. Then
\[
\Ext^l_{R_\Q}(\Q\otimes_\Z \Z_i, \Q\otimes_\Z \Z_j)  = 0
\]
for any $l\geq 1$ and $i, j \in I$. But 
\[
\Q\otimes_\Z \Ext^l_{R}(\Z_i, \Z_j) \simeq \Ext^l_{R_\Q}(\Q\otimes_\Z \Z_i, \Q\otimes_\Z \Z_j)
\]
(see e.g.\ \cite{weibel} Proposition 3.3.10) and so $\Ext^l_{R}(\Z_i, \Z_j)$ is torsion for any $l\geq 1$ and $i, j \in I$. Since it is also finitely generated, it is finite. The result for $\Tor^R_l(\Z_i, \Z_j)$ follows similarly.
\end{proof}

Let $R \subset \Gh(I)$ be a $B$-ring and $M$ a finitely generated $R$-module. Let $\mathcal{F}_\bullet$ be a free $R$-module resolution of $M$, where $\mathcal{F}_l = R^{\oplus n_l}$ for $l \geq 0$. Applying $\Hom_{R}(-, \Z_j)$ for some $j \in I$ gives a chain complex where each term is of the form $\Hom_{R}(R^{\oplus n_l}, \Z_j) \simeq \Z_j^{\oplus n_l}$. Applying $-\otimes_{R} \Z_j$ gives a chain complex where each term is of the form $R^{\oplus n_l}\otimes \Z_j \simeq \Z_j^{\oplus n_l}$. $\Ext^l_{R}(M, \Z_j)$ and $\Tor^{R}_l(M, \Z_j)$ are then isomorphic to subquotients of  $\Z_j^{\oplus n_l}$, and it follows that the $R$-module structure of each $\Ext^l_{R}(M, \Z_j)$ and $\Tor^{R}_l(M, \Z_j)$ is given by $r \in R$ acting by $r(j)$. It follows that for $i, j \in I$ and $l \geq 0$, any direct sum decomposition of $\Ext^l_{R}(\Z_i, \Z_j)$ or $\Tor^{R}_l(\Z_i, \Z_j)$ as an abelian group is automatically a decomposition as an $R$-module. 

By instead considering the functor $\Z_i \otimes_R -$ for $i \in I$, we note that for any $R$-module $N$, the $R$-module structure of $\Tor^R_l(\Z_i, N)$ for each $l \geq 0$ is given by $r$ acting by $r(i)$. Similarly, by considering an injective resolution of $N$, we have that the $R$-module structure of each $\Ext^l_{R}(\Z_i, N)$ is given by $r \in R$ acting by $r(i)$. 

\begin{mydef}For distinct $i, j \in I$, define $d(i, j)$ to be the greatest positive integer such that $r(i) \equiv r(j) \mod d(i, j)$ for each $r \in R$. 
\end{mydef}

Since the $R$-module structure of each $\Ext^l_{R}(\Z_i, \Z_j)$ and $\Tor^{R}_l(\Z_i, \Z_j)$ is given by both $r$ acting by $r(i)$ and $r$ acting by $r(j)$, it follows that each indecomposable $R$-module summand of $\Ext^l_{R}(\Z_i, \Z_j)$ and $\Tor^{R}_l(\Z_i, \Z_j)$ must be of the form $\Z_i / m\Z_i \simeq \Z_j / m \Z_j$ for some $m \mid d(i, j)$. 

Fix some rational prime $p$ and put $k = \F_p$. For a commutative ring $S$ write $\overline S$ for the quotient ring $S / pS \simeq S \otimes_{\Z}k$. For an $S$-module $M$, write $\overline{M}$ for the $\overline{S}$-module $M / pM$. If $M$ is annihilated by $p$, we will also denote the associated $\overline{S}$-module by $M$.

\begin{lem}\label{replace}
Let $S$ be a commutative ring which is free as a $\Z$-module. Let $M$ be a torsion-free $S$-module and $N$ an $S$-module annihilated by $pS$. Then
\[
\Ext^l_{S}(M, N) \simeq \Ext^l_{\overline{S}}(\overline{M}, N)
\]
and
\[
\Tor^{S}_l(M, N) \simeq \Tor^{\overline{S}}_l(\overline{M}, N)\]
for each $l \geq 0$.
\end{lem}

\begin{proof}
For any $S$-module $X$, any homomorphism of $S$-modules $\phi: X \to N$ must vanish on $pX$, and so we have an induced homomorphism $\overline{\phi}:\overline{X} \to N$. Similarly, any homomorphism of $\overline{S}$-modules $\psi: \overline{X} \to N$ lifts to a homomorphism of $S$-modules $X \to N$. It follows that 
\[
\Hom_{S}(X, N) \simeq \Hom_{\overline{S}}(\overline{X}, N).
\]

Let $(\mathcal{F}_\bullet, \partial_\bullet)$ be a free $S$-module resolution of $M$. In particular $\mathcal{F}_\bullet$ is a free $\Z$-module resolution of the $\Z$-module $M$, so applying $-\otimes_\Z k$ gives a chain complex over $\overline{M}$ with homology groups $\Tor^\Z_l(M, k)$. But $M$ is torsion-free so the homology groups vanish and the chain complex is exact. Since $\mathcal{F}_\bullet$ is a free $S$-module resolution, $\mathcal{F}_\bullet \otimes_\Z k$ is a free $\overline{S}$-module resolution.

Applying $\Hom_{\overline{S}}(-, N)$ to $\mathcal{F}_\bullet \otimes_\Z k$ and computing cohomology then computes the groups $\Ext^l_{\overline{S}}(\overline{M}, N)$. But $\Hom_{\overline{S}}(\overline{S}, N) \simeq \Hom_{S}(S, M)$, so this is the same as applying $\Hom_{S}(-, N)$ to $\mathcal{F}_\bullet$ and taking cohomology, i.e.\ computing the groups $\Ext^l_{S}(M, N)$. The proof for $\Tor^{S}_l(M, N)$ is analogous.
\end{proof}

For a $B$-ring $R \subset \Gh(I)$ and $i \in I$, write $k_i$ for the $R$-module where $r \in R$ acts on the field $k$ by $r(i)$. Put $\overline{R} = R\otimes_\Z k$.  

For each $j \in I$ we have a short exact sequence of $R$-modules
\[
0 \to \Z_j \xrightarrow{p} \Z_j \to k_j \to 0, \tag{$\dagger$}
\]
where the map $\Z_j \to \Z_j$ is given by multiplication by $p$.  Applying $\Hom_{R}(\Z_i, -)$ gives a long exact sequence beginning with
\begin{center}\begin{tikzpicture}[descr/.style={fill=white,inner sep=1.5pt}]
\matrix (m) [
            matrix of math nodes,
            ampersand replacement=\&,
            row sep=1em,
            column sep=2.5em,
            text height=1.5ex, text depth=0.25ex
        ]
        { 0 \&  \Hom_{{R}}(\Z_i, \Z_j) \& \Hom_{{R}}(\Z_i, \Z_j) \& \Hom_{\overline{{R}}}(k_i, k_{j}) \\
            \&  \Ext^1_{{R}}(\Z_i, \Z_j) \&  \Ext^1_{{R}}(\Z_i, \Z_j) \& \Ext^1_{\overline{{R}}}(k_i, k_j)  \\
            \& \Ext^2_{R}(\Z_i, \Z_j) \& \Ext^2_{R}(\Z_i, \Z_j) \& \ldots \\
        };

        \path[overlay,->, font=\scriptsize,>=latex]
        (m-1-1) edge (m-1-2)
        (m-1-2) edge (m-1-3) node[descr,yshift=1ex, xshift=10ex]  {$p$}
        (m-1-3) edge (m-1-4)
        (m-1-4) edge[out=355,in=175,] (m-2-2)
        (m-2-2) edge (m-2-3) node[descr,yshift=1ex, xshift=10ex]  {$p$}
        (m-2-3) edge (m-2-4)
        (m-2-4) edge[out=355,in=175,] (m-3-2)
	(m-3-2) edge (m-3-3) node[descr,yshift=1ex, xshift=10ex]  {$p$}
        (m-3-3) edge (m-3-4)
  
	;
	\end{tikzpicture}\end{center}
	
\noindent where we make use of the additivity of $\Ext^l_{R}(\Z_i, -)$ for each $l \geq 1$ to identify each map $\Ext^l_{R}(\Z_i, \Z_j) \to \Ext^l_{R}(\Z_i, \Z_j)$ as multiplication by $p$, and make use of Lemma \ref{replace} above to replace each $\Ext^l_{R}(\Z_i, k_j)$ with $\Ext^l_{\overline{R}}(k_i, k_j)$.	

For each $l\geq 1$ and each rational prime $q$, let $M_{l, q}$ be the submodule of $\Ext^l_{R}(\Z_i, \Z_j)$ annihilated by some power of $q$.  Since $\Ext^l_{R}(\Z_i, \Z_j)$ is finite, it follows that we have a decomposition of $R$-modules 
\[
\Ext^l_{R}(\Z_i, \Z_j) \simeq \bigoplus_q M_{l, q}
\]
where the sum is over all rational primes $q$.

For $l \geq 1$, let $a_l$ be the rank of $M_{l, p}$, and let $b_l$ be the $k$-dimension of $\Ext^l_{\overline{R}}(k_i, k_j)$. Note that $a_l$ is equal to the number of summands appearing in a decomposition of $M_{l, p}$ as a sum of indecomposable $R$-modules. For a non-zero indecomposable summand $\Z / p^\alpha \Z$ of $M_{l, p}$, the map $\Z / p^\alpha \Z \xrightarrow{p}\Z / p^\alpha \Z$ has kernel $p^{\alpha - 1} \Z / p^\alpha \Z \simeq \Z / p\Z$. So for $l \geq 1$, the kernel of the map $\left( \Ext^l_{R}(\Z_i, \Z_j) \xrightarrow{p} \Ext^l_{R}(\Z_i, \Z_j) \right)$ is a $k$-vector space with dimension $a_l$. Similarly, the cokernel of this map is a $k$-vector space of dimension $a_l$, and so the image of the connecting homomorphism $\Ext^l_{\overline{R}}(k_i, k_j) \to \Ext^{l+1}_{R}(\Z_i, \Z_j)$ has dimension $b_l - a_l$. Hence

\begin{align*}
a_{l+1} &=\dim_k \ker \left( \Ext^{l+1}_{R}(\Z_i, \Z_j) \xrightarrow{p} \Ext^{l+1}_{R}(\Z_i, \Z_j) \right) \\
&= \dim_k \im \left(  \Ext^l_{\overline{{R}}}(k_i, k_j) \to \Ext^{l+1}_{R}(\Z_i, \Z_j) \right) \\
&= b_l - a_l 
\end{align*}
for $l\geq 1$. In order to determine the sequence $a_l$, it is then sufficient to compute the sequence $b_l$. 

Applying $\Z_i \otimes_{R}-$ to $(\dagger)$ gives a long exact sequence ending with
\begin{center}\begin{tikzpicture}[descr/.style={fill=white,inner sep=1.5pt}]
\matrix (m) [
            matrix of math nodes,
            ampersand replacement=\&,
            row sep=1em,
            column sep=2.5em,
            text height=1.5ex, text depth=0.25ex
        ]
        {    \ldots \&\Tor_2^{R}(\Z_i, \Z_j) \& \Tor_2^{\overline{{R}}}(k_i, k_j) \& \\
             \Tor_1^{R}(\Z_i, \Z_j) \&  \Tor_1^{R}(\Z_i, \Z_j) \&\Tor_1^{\overline{{R}}}(k_i, k_j) \&  \\
             \Z_i\otimes \Z_j \& \Z_i\otimes \Z_j \& k_i \otimes \Z_j \& 0. \\
        };

        \path[overlay,->, font=\scriptsize,>=latex]
        (m-1-1) edge (m-1-2) 
        (m-1-2) edge (m-1-3) 
        (m-1-3) edge[out=355,in=175,] (m-2-1)
        (m-2-1) edge (m-2-2) node[descr,yshift=1ex, xshift=10ex]  {$p$}
        (m-2-2) edge (m-2-3)
        (m-2-3) edge[out=355,in=175,] (m-3-1)
	(m-3-1) edge (m-3-2) node[descr,yshift=1ex, xshift=10ex]  {$p$}
        (m-3-2) edge (m-3-3)
          (m-3-3) edge (m-3-4)

	;
	\end{tikzpicture}\end{center}
Write $z_l$ for the $p$-rank of $\Tor^{R}_l(\Z_i, \Z_j)$ and $y_l$ for the dimension $\Tor^{\overline{R}}_l(k_i, k_j)$. Then repeating the same argument as above gives 
\[
z_l = y_{l+1} - z_{l+1}
\]
for $l \geq 1$.
\newline

\begin{center}\emph{$B$-rings modulo a prime.}\end{center}

\noindent Since a $B$-ring $R\subset \Gh(I)$ is of finite rank as a $\Z$-module, $\overline{R}$ is a finite-dimensional $k$-algebra, and we have an $\overline{R}$-module decomposition of $\overline{R}$ as a direct sum of finitely many indecomposable projective $\overline{R}$-modules. Since $\overline{R}$ is commutative, this is a decomposition of commutative local $k$-algebras. By the usual block theory considerations (see e.g.\ \cite{auslander} Chapter II.5), studying the cohomology of $\overline{R}$ reduces to studying the cohomology of these indecomposable summands. 

Define a relation $\sim'_p$ on $I$ by putting $i \sim'_p j$ if and only if $p \mid d(i, j)$ for $i \neq j$. Note that by the definition of $d(i, j)$ this relation is symmetric and transitive, and we write $\sim_p$ for the equivalence relation defined by taking its reflexive closure. Let $\mathcal{E}$ denote the set of equivalence classes of $I$ with respect to $\sim_p$. For an equivalence class $E \in \mathcal{E}$, write $k_E$ for the (well-defined) $R$-module which is $k$ as an abelian group and where $r\cdot m = r(i)m$ for $m \in k$ and where $i$ is any element of $E$. 

\begin{lem}\label{replits}
For each $E \in \mathcal{E}$ there exists an $r \in R$ such that $r(i) \equiv 1 \mod p$ for each $i \in E$ and $r(j) \equiv 0 \mod p$ for each $j \notin E$. 
\end{lem}

\begin{proof}
For each equivalence class $E'$ distinct from $E$, we have 
an $r_{E'} \in R$ with $r_{E'}(i) \not\equiv r_{E'}(j) \mod p$ for $i \in E$ and $j \in E'$. Subtracting $r_{E'}(j)$ if required, we can assume $r_{E'}(j) = 0$, and hence $p \nmid r_{E'}(i)$. Replacing $r_{E'}$ by $r_{E'}^{p-1}$ if required, we can assume $r_{E'}(i) \equiv 1 \mod p$. Then putting $r = \prod_{\substack{E' \in \mathcal{E} \\ E' \neq E}}r_{E'}$ it is clear that $r$ has the claimed properties.
\end{proof}

\begin{pro}\label{rbar}
We have a surjective homomorphism of $k$-algebras
\[
\theta: \overline{R} \to \prod_{E \in \mathcal{E}} k_E
\]
with kernel the radical of $\overline{R}$. 
\end{pro}

\begin{proof}
We have a homomorphism of rings 
\[
R \to \prod_{E \in \mathcal{E}} k_E
\]
given by 
\[
r \mapsto (r(i) \mod p)_{E \in \mathcal{E}}
\]
where $i \in E$. By Lemma \ref{replits} this homomorphism is surjective. The kernel of this map contains $pR$, and we let $\theta$ be the induced surjective homomorphism of $k$-algebras. Since $\prod_{E \in \mathcal{E}} k_E$ is semisimple, the kernel of $\theta$ certainly contains the radical of $\overline{R}$. It remains to show the reverse inclusion. 

As in the proof of Proposition \ref{intrinsic}, choose for each $i \in I$ an element $s_i \in R$ such that $s_i(j) \neq 0$ if and only if $j = i$. Let $s_i(i) = p^{t_i}n_i$ where $n_i$ is coprime to $p$; put $t = \max_i t_i$ and put $N = \prod_{i\in I} n_i$. Let $r \in R$ be such that $r(i) \equiv 0 \mod p$ for each $i \in I$. Putting 
$q = Nr^{t+1}$, we have that $ps_i(i) \mid q(i)$ for each $i \in I$, and so we can define integers $m_i \in \Z$ by requiring $q(i) = pm_i s_i(i)$. It follows that 
\[
q = p\cdot \sum_{i \in I} m_is_i
\]
and so $q \in pR$. Then the image of $q$ in $\overline{R}$ is zero, and hence the image of $r^{t+1}$ in $\overline{R}$ is zero, since $N$ is coprime with $p$. 

It follows that the kernel of $\theta$ is nilpotent, and hence equal to the radical of $\overline{R}$.
\end{proof}

\begin{cor}\label{veryusefulcor}
\begin{enumerate}[i.]
\item $\overline{R}$ is the direct sum of $|\mathcal{E}|$ indecomposable $k$-algebras;
\item the set $\{k_E\}_{E \in \mathcal{E}}$ is a complete irredundant set of irreducible modules for $\overline{R}$;
\item the dimension of the indecomposable summand corresponding to $k_i$ is $|E|$, the cardinality of the $\sim_p$-equivalence class of $i \in I$. The maximal ideal of each indecomposable summand has codimension 1.
\end{enumerate}
\end{cor}

\begin{proof}
The only part that does not follow immediately is iii. For an equivalence class $E$, let $R'$ be the $B$-ring $R' \subset \prod_{i \in E}\Z$ induced by $R$, and $\pi_E: R \to R'$ the corresponding homomorphism. Then $\pi_E$ descends to a map $\overline{R} \to \overline{R'}$ and this is a surjection of $k$-algebras. Applying part i to $\overline{R'}$, the $k$-algebra $\overline{R'}$ is indecomposable of dimension $|E|$, and so the indecomposable summand of $\overline{R}$ corresponding to $E$ has dimension $\geq |E|$. Since we can do this for each equivalence class in $\mathcal{E}$, the summand must have dimension $|E|$. Since the radical of $\overline{R}$ has dimension $\dim R - |\mathcal E|$, the radical of the summand corresponding to $E$ must have dimension $|E| - 1$.
\end{proof}

It follows that each indecomposable summand of $\overline{R}$ is a commutative local finite-dimensional $k$-algebra $S$ with maximal ideal $\mathcal{M}$ satisfying $S /\M \simeq k$. There is a unique $S$-module structure on $k$ where $\M \cdot k = 0$, and we denote this $S$-module by $k$. The following result is standard (see \cite{torbook} Chapter 2, \S 3).

\begin{lem}\label{exttor}
Let $S$ be as above. Then $\Tor^{S}_l(k, k) \simeq \Ext^l_{S}(k, k)$ as $S$-modules for each $l \geq 0$.
\end{lem}

\begin{cor}\label{exttorcomp}
For $i, j \in I$, write $a_l$ for the $p$-rank of $\Ext_R^l(\Z_i, \Z_j)$ and $z_l$ for the $p$-rank of $\Tor^R_l(\Z_i, \Z_j)$. Then 
$
z_{l} = a_{l+1}
$
for all $l \geq 1$.
\end{cor}

\begin{proof}
We have recurrence relations 
\[
a_{l+1} = b_l - a_l
\]
and 
\[
z_l = y_{l+1} - z_{l+1}
\]
for $l \geq 1$, where $b_l$ and $y_l$ are the dimensions of the $k$-vector spaces $\Ext^l_{\overline{R}}(k_i, k_j)$ and $\Tor^{\overline{R}}_l(k_i, k_j)$ respectively. If $i \not\sim_p j$ then $\Ext^l_{\overline{R}}(k_i, k_j) = \Tor^{\overline{R}}_l(k_i, k_j) = 0$ for each $l$ and the result is trivially true. Otherwise, by Lemma \ref{exttor} we have $b_l = y_l$. 

Suppose $i = j$. It is clear that $\Hom_R(\Z_i, \Z_i) \simeq \Z_i$ and $\Hom_R(\Z_i, k_i) \simeq k_i$ so the long exact sequence associated to $(\dagger)$ begins with the short exact sequence
\[
0 \to \Hom_R(\Z_i, \Z_i) \xrightarrow{p} \Hom_R(\Z_i, \Z_i) \to \Hom_R(\Z_i, k_i) \to 0.
\]
It follows that the map $\Ext^1_R(\Z_i, \Z_i) \xrightarrow{p} \Ext^1_R(\Z_i, \Z_i)$ is injective and hence $\Ext^1_R(\Z_i, \Z_i)$ has zero $p$-part and $a_1 = 0$. Similarly, $z_1 = b_1$, and the recurrence gives $z_l = a_{l+1}$ for all $l\geq 1$ as claimed.

Suppose $i \neq j$ with $i \sim_p j$. It is clear that $\Hom_R(\Z_i, \Z_j) = 0$ and $\Hom_R(\Z_i, k_j) \simeq k_j$, so by inspection of the long exact sequence associated to $(\dagger)$ we have $a_1 = 1$. Similarly we have $\Z_i \otimes \Z_j \simeq \Z_j / d(i, j)\Z_j$ and $\Z_i \otimes k_j \simeq k_j$ from which we obtain $z_1 = b_1 -1$, and once more we have $z_l = a_{l+1}$ for each $l \geq 1$. 
\end{proof}

\begin{cor}\label{semisimplecor}
Suppose $p \nmid d(i, j)$ for all distinct $i, j \in I$. Then $\overline{R}$ is semisimple.
\end{cor}

\begin{proof}
Since $p \nmid d(i, j)$ for all distinct $i, j \in I$, we have $|\mathcal{E}| = |I|$ and the homomorphism $\theta$ of Proposition \ref{rbar} is an isomorphism.
\end{proof}

\begin{cor}\label{notmanylabelsleft}
Let $i, j \in I$ be distinct with $d(i, j) = 1$. Then 
\[
\Ext^l_R(\Z_i, \Z_j)  = 0
\]
for all $l \geq 0$.
\end{cor}

\begin{proof}
The case $l=0$ is immediate.

Since $p \nmid d(i, j)$, the simples modules $k_i$ and $k_j$ are non-isomorphic and therefore belong to different blocks of the commutative $k$-algebra $\overline{R}$. Then $\Ext^l_{\overline{R}}(k_i, k_j) = 0$ for each $l \geq 0$, and so the $p$-part of $\Ext^l_R(\Z_i, \Z_j)$ is zero for each $l \geq 0$. 
Since this is true of any prime $p$, and since $\Ext^l_R(\Z_i, \Z_j)$ is finite for $l \geq 1$, it follows that $\Ext^l_R(\Z_i, \Z_j) = 0$ for all $l \geq 1$. \end{proof}

\begin{cor}\label{unicor}
Suppose that distinct elements $i, j$ of $I$ are such that $\{i, j\}$ is an equivalence class for the relation $\sim_p$ on $I$. Write $a_l$ and $a'_l$ for the $p$-ranks of $\Ext^l_R(\Z_i, \Z_i)$ and $\Ext^l_R(\Z_i, \Z_j)$ respectively. Then
\[
a_l = \begin{cases}  0&\text{ if $l$ odd} \\ 1 &\text{ if $l$ even} \end{cases}
\]
\[
a'_l = \begin{cases}  1&\text{ if $l$ odd} \\ 0 &\text{ if $l$ even} \end{cases}
\]
for all $l \geq 1$.
\end{cor}

\begin{proof}
The algebra $\overline{R}$ has an indecomposable 2-dimensional $k$-algebra summand corresponding to the pair $\{i, j\}$. Any indecomposable local $k$-algebra of dimension 2 with maximal ideal of codimension 1 is isomorphic to $A = k[x]/(x^2)$. We have a free $A$-module resolution of $k$ given by 
\[
\ldots \to A \to \ldots \to A \to A \to k \to 0
\]
where each map $A \to A$ is given by $1_A \mapsto x$, and so $\Ext^l_A(k, k) \simeq k$ for all $l \geq 0$, i.e.\ in the notation of our recurrence relations we have $b_l= 1$ for all $l \geq 1$. Now $a_{l+1} = b_l - a_l$ and $a'_{l+1} = b_l - a'_l$, where $a_1 = 0$ and $a'_1 = 1$, from which the corollary follows immediately.
\end{proof}
 { }~ \newline

\begin{center}\emph{The Burnside ring as $B$-ring.}\end{center}

\begin{pro}
The embedding $\pi: A(G) \to \Gh(\ccs(G))$ defines a $B$-ring.
\end{pro}

\begin{proof}
We need to show that for non-conjugate subgroups $H, J \subset G$ we can find an $a \in A(G)$ with $\pi_H(a) \neq 0$ and $\pi_J(a) = 0$. Now if $J$ is not conjugate to a subgroup of $H$, then $a=[G/H]$ suffices. Otherwise, if $J$ is conjugate to a subgroup of $H$ then $H$ is not conjugate to a subgroup of $J$, and putting $a = [N_GJ:J][G/G] - [G/J]$ we have $\pi_J(a) = 0$ and $\pi_H(a) = [N_GJ:J] \neq 0$.
\end{proof}

For $H \subset G$ and $p$ a prime, let $O^p(H)$ denote the smallest normal subgroup $K \triangleleft H$ such that $H/K$ is a $p$-group. We recall the following result due to Dress (\cite{dress} Proposition 1).

\begin{pro}\label{dresspro}
Let $H, J$ be subgroups of $G$. Then $\pi_H(a) \equiv \pi_J(a) \mod p$ for each $a \in A(G)$ if and only if $O^p(H)$ is conjugate to $O^p(J)$. 
\end{pro}

Since $A(G)$ is a $B$-ring, we have integers $d((H), (J))$ defined for each distinct pair $(H), (J) \in \ccs(G)$. For $H, J \subset G$ non-conjugate subgroups, we write $d(H, J)$ for $d((H), (J))$. It follows immediately from Proposition \ref{dresspro} that $p \mid d(H, J)$ if and only if $O^p(J)$ is conjugate to $O^p(H)$. 

\begin{cor}
Let $H, J \subset G$ be subgroups such that $O^p(H)$ is not conjugate to $O^p(J)$ for any rational prime $p$. Then 
\[
\Ext^l_{A(G)}(\Z_H, \Z_J) = 0
\]
for all $l \geq 0$.
\end{cor}

\begin{lem}\label{newandnifty}
$d(H, J) \mid |G|$ for each pair $H, J$ of non-conjugate subgroups of $G$.
\end{lem}

\begin{proof}
Without loss of generality we can suppose that $H$ is not conjugate to a subgroup of $J$. Then $\pi_H([G/J]) = 0$ and $\pi_J([G/J]) = [N_GJ:J]$. It follows that $d(H, J) \mid [N_GJ:J]$ and hence $d(H, J) \mid |G|$. 
\end{proof}

\begin{cor}\label{lastlabel}
Suppose $p\mid |G|$ and $p^2 \nmid |G|$. Let $\mathcal{E}$ be the set of equivalence classes of the relation $\sim_p$ on $\ccs(G)$, and let $E_1, \ldots , E_q$ be the equivalence classes of cardinality 2. For $H \subset G$ and $l\geq 1$, write $M_{H, l}$ for the $p$-part of $\Ext^l_{A(G)}(\Z_H, \Z_H)$; for $J \subset G$ not conjugate to $H$ and $l \geq 1$, write $N_{H, J, l}$ for the $p$-part of $\Ext^l_{A(G)}(\Z_H, \Z_J)$. Then
\begin{enumerate}[i.]
\item each $E \in \mathcal{E}$ has cardinality $\leq 2$;
\item $M_{H, l}$ is non-zero for some $l \geq 1$ if and only if $(H) \in E_i$ for some $1 \leq i \leq q$;
\item $N_{H, J, l}$ is non-zero for some $l \geq 1$ if and only if $\{(H), (J)\} = E_i$ for some $1\leq i \leq q$;
\item if $(H) \in E_i$ for some $1 \leq i \leq q$ then 
\[
M_{H,l} \simeq  \begin{cases} 0  &\text{ if $l$ odd} \\ \Z / p\Z &\text{ if $l$ even}; \end{cases}
\]
\item if $\{(H), (J)\} = E_i$ for some $1\leq i \leq q$, then 
\[
N_{H, J, l} \simeq \begin{cases}  \Z / p\Z &\text{ if $l$ odd} \\ 0 &\text{ if $l$ even}. \end{cases}
\]
\end{enumerate}
\end{cor}

\begin{proof}
Suppose $E$ is an equivalence class of cardinality $\geq 3$, and consider  distinct conjugacy classes of subgroups $(H), (J) , (J')$ in $E$. Without loss of generality we can assume $O^p(J) = O^p(J') = H$. Then $J/H$ and $J'/H$ are $p$-groups in $N_GH/H$. Since $p^2 \nmid |G|$, we certainly have $p^2 \nmid |N_GH/H|$, so $J/H$ and $J'/H$ are Sylow $p$-groups in $N_GH/H$. Then $J/H$ and $J'/H$ are conjugate in $N_GH/H$, and so $J$ and $J'$ are conjugate in $G$, a contradiction. So each class $E \in \mathcal{E}$ has cardinality at most 2.

For parts ii and iii, we know by Corollary \ref{unicor} that if $(H) \in E_i$ for some $i$ then $M_{H, l}$ is non-zero whenever $l$ is even, and if $\{(H), (J)\} = E_i$ for some $i$ then $N_{H, J, l}$ is non-zero whenever $l$ is odd. Otherwise $(H) \notin E_i$ for each $i$, in which case by part i $\{(H)\}$ is an equivalence class in $\mathcal{E}$ and so the block corresponding to $H$ is 1-dimensional and $M_{H, l} = 0$ for all $l \geq 1$. Similarly, if $\{(H), (J)\} \neq E_i$ for each $i$, then $(H)$ and $(J)$ belong to distinct equivalence classes and $N_{H, J, l} = 0$ for all $l \geq 0$. 

Suppose $\{(H), (J)\}$ is an equivalence class in $\mathcal{E}$. By Corollary \ref{unicor}, $M_{H, l}$ has a $p$-power summand only when $l$ is even, and $N_{H, J, l}$ has a $p$-power summand only when $l$ is odd. It remains to show that each $p$-power summand is in fact $\Z / p\Z$. For $N_{H, J, l}$ this is clear since $d(H, J)$ annihilates $N_{H, J, l}$ and $p^2 \nmid d(H, J)$ by Lemma \ref{newandnifty}. For $M_{H, l}$ note that by dimension shifting this is the $p$-part of $\Ext^{l-1}_{A(G)}(K_H, \Z_H)$ where $K_H = \ker \pi_H$. Since $p^2 \nmid d(H, J')$ for any $J' \subset G$, and since $p \mid d(H, J')$ if and only if $(J')=(J)$, we can construct an $s_H \in A(G)$ such that $\pi_{J'}(s_H) \neq 0$ if and only if $(H) = (J')$, and such that $p^2 \nmid \pi_H(s_H)$. Now $s_H$ annihilates $K_H$, so it follows that $s_H$ annihilates $\Ext^l_{A(G)}(\Z_H, \Z_H)$, and so a fortiori $s_H$ annihilates $M_{H, l}$. So any $p$-power summand of $M_{H, l}$ must be of the form $\Z / p\Z$.
\end{proof}

\begin{mythe}\label{squarefree}
Suppose $|G|$ is square-free. Then for all $H, J \subset G$ and $l \geq 1$, we have an isomorphism of $A(G)$-modules $\Ext^l_{A(G)}(\Z_H, \Z_J) \simeq \Ext^{l+2}_{A(G)}(\Z_H, \Z_J)$.
\end{mythe}

In the remainder we establish the converse, by showing that if $|G|$ is not square-free then there exists $H, J \subset G$ such that the rank of $\Ext^l_R(\Z_H, \Z_J)$ is unbounded as $l \to \infty.$ This is an easy consequence of the following two results. Recall that for a field $F$, a commutative $F$-algebra $S$ is said to be symmetric if there exists an $F$-linear map $\l: S \to F$ such that $\ker \l$ contains no non-zero ideals of $S$.

\begin{mythe}[Gustafson \cite{gustafson}]\label{gusta}
If $p^2 \mid |G|$ and $F$ is a field of characteristic $p$ then the $F$-algebra $A(G) \otimes_\Z F$ is not symmetric. 
\end{mythe}

\begin{mythe}[Gulliksen \cite{gulliksen}]\label{gulli}
Let $S$ be a commutative noetherian local ring with maximal ideal $\M$ and with residue field $F=S / \M$. Then the sequence $(\dim\Tor^S_l(F, F))_{l \in \N}$ is bounded if and only if 
\[
d(S) \geq \dim\M/\M^2 - 1,
\]
where $d(S)$ denotes the Krull dimension of $S$.
\end{mythe}

\begin{mythe}
If $|G|$ is not square-free then there exist subgroups $H, J \subset G$ such that the groups $\Ext^l_{A(G)}(\Z_H, \Z_J)$ have unbounded rank.
\end{mythe}

\begin{proof}
Let $p$ be a prime with $p^2 \mid |G|$ and put $k = \F_p$.
The algebra $\overline{A(G)} = A(G) \otimes_\Z k$ is not symmetric by Theorem \ref{gusta}, so it has an indecomposable $k$-algebra summand $S$ which is not symmetric. Now $S$ is finite-dimensional so the maximal ideal $\M$ of $S$ is nilpotent. Then $\M$ is the only prime of $S$ and $d(S) = 0$. It follows by Theorem \ref{gulli} that the sequence $(\dim\Tor^S_l(k, k))_{l \in \N}$ is unbounded if and only if $\dim \M / \M^2 >1$.

If $\dim \M / \M^2 = 0$ then $S = k$ is clearly symmetric. Suppose $\dim \M / \M^2 = 1$. Let $t \in S$ generate $\M$ and note that $\{1_S, t, \ldots , t^q\}$ is a vector space basis for $S$ for some $q \geq 1$. Define $\l: S \to k$ by putting
\[
\l\left(\sum_{i = 0}^q a_it^i\right) = a_q.
\]
Now if $J \subset S$ is a non-zero ideal then we can choose some non-zero element $s = \sum_{i = 0}^q b_it^i$ in $J$, and choose $m$ to be minimal such that $b_m \neq 0$. Then $t^{q-m}s = b_mt^q \in J$ is not in $\ker \l$, so $\ker \l$ contains no non-zero ideals and $S$ is symmetric, a contradiction.
It follows that $\dim \M / \M^2 > 1$ and the sequence $(\dim\Tor^S_l(k, k))_{l \in \N} = (\dim\Ext^l_S(k, k))_{l \in \N}$ is unbounded.

By Corollary \ref{veryusefulcor}, the summand $S$ corresponds to some equivalence class $E \in \mathcal{E}$, and we have 
\[
\dim\Ext^l_S(k, k) = \dim\Ext^l_{\overline{A(G)}}(k_E, k_E).
\]
Let $H, J$ be subgroups of $G$ with $(H), (J) \in E$,
let $\alpha_l$ be the $p$-rank of $\Ext^l_{A(G)}(\Z_H, \Z_J)$, and let $\beta_l$ be the dimension of $\Ext^l_{\overline{A(G)}}(k_E, k_E)$. We have a recurrence 
\[
\alpha_{l+1} = \beta_l - \alpha_l
\]
for $l\geq 1$.
Since $\beta_l$ is unbounded, it follows immediately that $\alpha_l$ is unbounded, and hence the groups $\Ext^l_{A(G)}(\Z_H, \Z_J)$ have unbounded rank.
\end{proof}

 {}

\end{document}